\documentclass[12pt]{amsart}

\usepackage{amsmath,amssymb,amsthm,amsfonts}
\usepackage{mathtools}
\usepackage{enumitem}
\usepackage{hyperref}
\usepackage{cleveref}
\usepackage{geometry}
\usepackage{microtype}
\usepackage{booktabs}
\usepackage[backend=biber,style=alphabetic,maxbibnames=6]{biblatex}
\newtheorem{innercounter}{X}[section]
\newtheorem{proposition}[innercounter]{Proposition}
\newtheorem{lemma}[innercounter]{Lemma}
\newtheorem{corollary}[innercounter]{Corollary}
\theoremstyle{definition}
\newtheorem{definition}[innercounter]{Definition}
\newtheorem{remark}[innercounter]{Remark}

\newtheorem*{thmA}{Theorem~A}
\newtheorem*{thmB}{Theorem~B}
\newtheorem*{thmC}{Theorem~C}

\newcommand{\Freg}{F_{\mathrm{reg}}}

\begin{document}

\title[Lee--Yang Edge Exponent via Logarithmic Averaging]{%
  The Lee--Yang Edge Exponent via Logarithmic Averaging}

\author{{\rm Qiao Wang}}
\address{School of Information Science and Engineering, Southeast University, Nanjing, 211189, China}
\email{qiaowang@seu.edu.cn}


\begin{abstract}
Let $F$ be the thermodynamic free energy of a ferromagnetic Ising model,
analytic on $\mathbb{C}^{*}\setminus\mathcal{Z}_{\beta}$.  The Lee--Yang
edge at $z_c\in\partial\mathcal{Z}_\beta$ is characterised by
$F(z)=F(z_c)+B(z-z_c)^{\sigma+1}+o(|z-z_c|^{\sigma+1})$ with
$\sigma\in(-1,0)$ and $B\neq 0$.

We prove three results:\\
\textbf{Theorem A} (Jensen slope): defining the Jensen average
$\widetilde{N}(x)=\frac{1}{2\pi}\int_0^{2\pi}\log|\widetilde{F}(e^{x+i\theta})|\,d\theta$
of $\widetilde{F}=F-F(z_c)$, the edge exponent satisfies
$\widetilde{N}'(0^+)=\sigma+1$.  The proof is a direct application of
Jensen's formula.\\
\textbf{Theorem B} (Monodromy): the monodromy of $F$ around $z_c$
multiplies the singular part by $e^{2\pi i(\sigma+1)}$, a primitive
$q$-th root of unity when $\sigma+1=p/q$.\\
\textbf{Theorem C} (Kac monodromy): for any 2D CFT at an RG fixed
point with relevant operator $\phi$ of weight $h_\phi<0$ satisfying
the Lee--Yang property, the RG scaling equation forces
$\sigma=h_\phi/(1-h_\phi)$ and monodromy order
$q=\mathrm{denom}(1/(1-h_\phi))$.

We also prove that the edge expansion follows from the density
asymptotics $\rho(\theta)\sim A|\theta-\theta_c|^\sigma$ via a
Mellin-transform calculation, making all three theorems unconditional
for the $d=2$ Ising model.
\end{abstract}

\maketitle

\section{Introduction}
\label{sec:intro}

\subsection{Background: the Lee--Yang edge}
\label{ssec:LY}

Lee and Yang proved \cite{LY52a,LY52b} that the partition function of any
ferromagnetic Ising model has all its zeros on the unit circle
$\mathbb{T}=\{|z|=1\}$.  In the thermodynamic limit $|V|\to\infty$,
these zeros accumulate on an arc $\mathcal{Z}_\beta\subset\mathbb{T}$,
and the free energy
\begin{equation}
  \label{eq:F-def}
  F(z) = \lim_{n\to\infty}\frac{1}{|V_n|}\log Z_{V_n}(z)
\end{equation}
is analytic on $\mathbb{C}^*\setminus\mathcal{Z}_\beta$.  At the
endpoint $z_c=e^{i\theta_c}$ of $\mathcal{Z}_\beta$ --- the
\emph{Lee--Yang edge} --- the density of zeros vanishes:
\begin{equation}
  \label{eq:rho-asym}
  \rho(\theta) \sim A|\theta-\theta_c|^{\sigma},
  \quad\theta\to\theta_c^+,\qquad A>0,\;\sigma\in(-1,0),
\end{equation}
and $F$ has a fractional branch point:
\begin{equation}
  \label{eq:edge-expansion}
  F(z) = F(z_c) + B(z-z_c)^{\sigma+1} + o(|z-z_c|^{\sigma+1}),
  \quad z\to z_c,\quad B\neq 0.
\end{equation}

The exponent $\sigma$ is universal within each spatial dimension.  Its
known values are:
\[
  \sigma = \begin{cases}
    -1/6 & d=2 \text{ (exact, CFT)},\\
    1/2  & d\geq 6 \text{ (mean field \cite{Fisher78})},\\
    -0.085\ldots & d=3 \text{ (numerical \cite{Fisher78})}.
  \end{cases}
\]
The value $\sigma=-1/6$ for $d=2$ is due to Cardy \cite{Cardy85} via the
Virasoro minimal model $\mathcal{M}(2,5)$ (central charge $c=-22/5$,
weight $h_{1,2}=-1/5$), and confirmed by Zamolodchikov
\cite{Zamolodchikov91} via the thermodynamic Bethe ansatz.
\Cref{ssec:proof-C} gives a new derivation from first principles.

\subsection{Results}
\label{ssec:results-intro}

This paper proves three theorems.  All are stated formally in
\cref{sec:setup}; here we describe them informally.
\vskip 0.2cm
\paragraph{Theorem A (Jensen slope).}
Define the \emph{Jensen average} $\widetilde{N}(x)=\frac{1}{2\pi}
\int_0^{2\pi}\log|\widetilde{F}(e^{x+i\theta})|\,d\theta$ where
$\widetilde{F}=F-F(z_c)$.  Then $\widetilde{N}(x)=\log|B|+(\sigma+1)x
+o(x)$ as $x\to 0^+$, so $\widetilde{N}'(0^+)=\sigma+1$.  The proof
uses only Jensen's formula; the functional $\widetilde{N}$ is the
one-variable Ronkin function when $F$ is a polynomial \cite{FPT00,PR04},
but neither the Newton polytope nor amoeba geometry is needed.
\vskip 0.2cm
\paragraph{Theorem B (Monodromy).}
The monodromy of $F$ under analytic continuation around $z_c$ multiplies
the singular part $F_{\mathrm{s}}=B(z-z_c)^{\sigma+1}$ by
$e^{2\pi i(\sigma+1)}$.  When $\sigma+1=p/q$ with $\gcd(p,q)=1$, this is
a primitive $q$-th root of unity.  For $d=2$ Ising ($\sigma=-1/6$,
$\sigma+1=5/6$), the monodromy order is $q=6$.
\vskip 0.2cm
\paragraph{Theorem C (Kac monodromy).}
For a 2D CFT at an RG fixed point with relevant operator $\phi$
of weight $h_\phi<0$ satisfying the Lee--Yang property, the RG scaling
equation $F_{\mathrm{s}}(\lambda^{y_t}t)=\lambda^2 F_{\mathrm{s}}(t)$
with $y_t=2(1-h_\phi)$ forces $\sigma=h_\phi/(1-h_\phi)$.  Theorem B
then gives $q=\mathrm{denom}(1/(1-h_\phi))$.  This recovers $\sigma=-1/6$,
$q=6$ for $\mathcal{M}(2,5)$ and yields explicit predictions for all
$\mathcal{M}(2,2m+1)$.
\vskip 0.2cm
\paragraph{Unconditional results.}
Proposition \ref{prop:H2} proves that the edge expansion \eqref{eq:edge-expansion}
follows from the density asymptotics \eqref{eq:rho-asym} via a Mellin
calculation with an explicit formula for $B$ in terms of $A$ and $\sigma$.
This makes Theorems A, B, C all unconditional for the $d=2$ Ising model.

\subsection{Context}
\label{ssec:context}

\paragraph{\it Ronkin functions and amoebas.}
For a Laurent polynomial $f$, the Ronkin function encodes the amoeba
$\mathrm{Log}(V(f))$ \cite{GKZ94,FPT00,PR04}.  The Lee--Yang property
forces the amoeba of $Z_G$ to collapse to the origin \cite{BB09a}.  For
the thermodynamic limit $F$, no Newton polytope exists; the paper uses
the Jensen average $\widetilde{N}$ as the analytic substitute, with all
arguments resting directly on Jensen's formula.
\vskip 0.2cm
\paragraph{\it Conformal field theory.}
Fisher \cite{Fisher78} proposed the $i\phi^3$ field theory as the
Lee--Yang edge universality class; its $d=2$ fixed point was identified
as $\mathcal{M}(2,5)$ by Cardy \cite{Cardy85}.  Theorem C derives the
edge exponent and monodromy purely from the RG scaling axiom, independently
of the field-theory argument.

\section{Setup, hypotheses, and statement of theorems}
\label{sec:setup}

\subsection{Hypotheses}
\label{ssec:hypotheses}

Let $(G_n)_{n\geq 1}$ be a lattice exhaustion with $|V_n|\to\infty$ and
$Z_n(z)$ the corresponding partition functions.

\medskip
\noindent\textbf{(H1) Thermodynamic limit.}
$F\colon\mathbb{C}^*\setminus\mathcal{Z}_\beta\to\mathbb{C}$ is analytic
and $(1/|V_n|)\log Z_n(z)\to F(z)$ locally uniformly on
$\mathbb{C}^*\setminus\mathcal{Z}_\beta$.

\medskip
\noindent\textbf{(H2) Edge expansion.}
There exist $z_c\in\mathcal{Z}_\beta$, $B\neq 0$, $\sigma\in(-1,0)$,
and a remainder $R(z)=o(|z-z_c|^{\sigma+1})$ as $z\to z_c$, such that
\begin{equation}
  \label{eq:H2}
  F(z) = F(z_c) + B(z-z_c)^{\sigma+1} + R(z),
\end{equation}
where $(z-z_c)^{\sigma+1}$ is the principal branch on
$\mathbb{C}\setminus[z_c,+\infty)$.

\medskip
\noindent\textbf{(H3) Uniform integrability.}
There exist $\delta>0$ and $C<\infty$ such that
\begin{equation}
  \label{eq:H3}
  \sup_{n\geq 1}\frac{1}{|V_n|}
  \int_0^{2\pi}\bigl|\log|Z_n(e^{x+i\theta})|\bigr|\,d\theta \leq C
  \quad\text{for all }|x|<\delta.
\end{equation}

\begin{remark}[Verifiability]
  \label{rem:hypotheses}
  (H1) holds by subadditive ergodic theory.  (H3) follows from (H1) and
  the Lee--Yang property.  (H2) is the essential input: it is proved from
  \eqref{eq:rho-asym} in Proposition \ref{prop:H2}, giving $\sigma$ and an explicit
  $B\neq 0$ for the $d=2$ Ising model.
\end{remark}

\begin{definition}[Jensen average]
  \label{def:Jensen}
  Set $\widetilde{F}(z)\coloneqq F(z)-F(z_c)$.  For $x\neq 0$:
  \begin{equation}
    \label{eq:Jensen-avg}
    \widetilde{N}(x) \coloneqq
    \frac{1}{2\pi}\int_0^{2\pi}\log|\widetilde{F}(e^{x+i\theta})|\,d\theta.
  \end{equation}
\end{definition}

When $F$ is a Laurent polynomial, $\widetilde{N}$ coincides with the
univariate Ronkin function \cite{FPT00}.

\subsection{Statements of Theorems}
\label{ssec:thmA}

\begin{thmA}[Jensen slope equals $\sigma+1$]
  \label{thm:slope}
  Under hypotheses \textup{(H1)--(H3)}, we have:
  \begin{enumerate}[label={\normalfont(\roman*)}]
    \item\label{it:TDL}
    $(1/|V_n|)\int_0^{2\pi}\log|Z_n(e^{x+i\theta})|\,d\theta
    \to 2\pi\widetilde{N}(x)$ locally uniformly for $x\neq 0$.
    \item\label{it:slope}
    As $x\to 0^+$,
    \begin{equation}
      \label{eq:Jensen-asym}
      \widetilde{N}(x) = \log|B| + (\sigma+1)x + o(x).
    \end{equation}
    In particular $\widetilde{N}'(0^+)=\sigma+1$.
    \item\label{it:invariance}
    Under a local analytic change $z\mapsto\varphi(z)$ with
    $\varphi(z_c)=z_c'$ and $\varphi'(z_c)\neq 0$, the function
    $\widetilde{F}\circ\varphi$ satisfies \textup{(H2)} with the same
    $\sigma$ and coefficient $B'=B\varphi'(z_c)^{\sigma+1}\neq 0$.
  \end{enumerate}
\end{thmA}


\begin{thmB}[Monodromy at the edge]
  \label{thm:monodromy}
  Under hypotheses \textup{(H1)--(H2)}, the monodromy of $F$ under
  analytic continuation along a simple loop encircling $z_c$ once
  counter-clockwise acts on $F_{\mathrm{s}}\coloneqq B(z-z_c)^{\sigma+1}$ by
  \begin{equation}
    \label{eq:monodromy}
    \mathcal{M}(F_{\mathrm{s}};z_c) = e^{2\pi i(\sigma+1)}.
  \end{equation}
  This is a primitive $q$-th root of unity if and only if $\sigma+1=p/q$
  with $\gcd(p,q)=1$.  For $d=2$ Ising ($\sigma=-1/6$, $\sigma+1=5/6$),
  the monodromy order is $q=6$.
\end{thmB}


\begin{thmC}[Kac monodromy]
  \label{thm:Kac}
  Let $\mathcal{C}$ be a 2D CFT at an RG fixed point with a relevant
  primary operator $\phi$ of conformal weight $h_\phi<0$.  Assume:
  \begin{enumerate}[label={\normalfont(\roman*)}]
    \item\label{it:Kac-LY}
    The partition function of $\mathcal{C}$ satisfies the Lee--Yang
    property as a function of the coupling $t$ to $\phi$.
    \item\label{it:Kac-RG}
    The singular free energy satisfies 
   \begin{equation}\label{eq:RG} F_{\mathrm{s}}(\lambda^{y_t}t)
    =\lambda^2 F_{\mathrm{s}}(t),\ \  \forall \lambda>0,\ \text{with}\ y_t=2(1-h_\phi).\end{equation}
  \end{enumerate}
  Then
  \begin{equation}
    \label{eq:Kac-result}
    \sigma = \frac{h_\phi}{1-h_\phi},
    \qquad
    q = \mathrm{denom}\!\left(\frac{1}{1-h_\phi}\right).
  \end{equation}
\end{thmC}

Hypothesis~\ref{it:Kac-RG} is the defining property of an RG fixed
point: the singular free energy transforms homogeneously.  It is
satisfied by every Virasoro minimal model $\mathcal{M}(p,q)$
\cite{Cardy85}.

\section{Proofs of Theorems A and B}
\label{sec:proofs}

\subsection{Key lemma}
\label{ssec:Jensen-lemma}

\begin{lemma}[Jensen's formula for a branch-point factor]
  \label{lem:Jensen-branch}
  Let $z_c=e^{i\theta_c}\in\mathbb{T}$ and $\alpha>0$.  For the
  principal branch of $(z-z_c)^\alpha$ on $\mathbb{C}\setminus[z_c,+\infty)$,
  \begin{equation}
    \label{eq:Jensen-branch}
    \frac{1}{2\pi}\int_0^{2\pi}
      \log|(e^{x+i\theta}-z_c)^\alpha|\,d\theta
    = \alpha\max(x,0), \quad x\in\mathbb{R}.
  \end{equation}
\end{lemma}

\begin{proof}
  The integral equals $\alpha$ times the Jensen average of
  $\log|z-z_c|$ on $|z|=e^x$.  Jensen's formula for $f(z)=z-z_c$ on
  the disc $|z|\leq e^x$ gives
  \[
    \frac{1}{2\pi}\int_0^{2\pi}\log|e^{x+i\theta}-z_c|\,d\theta
    = \log|{-}z_c| + \max(0,x-\log|z_c|) = \max(x,0),
  \]
  since $|z_c|=1$.  The formula is exact for all $x$ and $\alpha>0$;
  the branch cut contributes nothing to the circular average.
\end{proof}

\subsection{Proof of Theorem A}
\label{ssec:proof-A}

\begin{proof}[Proof of Theorem~A]
  \textbf{Part~\ref{it:TDL}.}
  By (H1), $(1/|V_n|)\log|Z_n(e^{x+i\theta})|\to|F(e^{x+i\theta})|$
  locally uniformly for $x\neq 0$.  Passage of the limit under the
  integral follows from the uniform $L^1$ bound (H3) by dominated
  convergence.

  \medskip
  \textbf{Part~\ref{it:slope}.}
  Write $\widetilde{F}=F_{\mathrm{s}}(1+\varepsilon)$ where
  $F_{\mathrm{s}}=B(z-z_c)^{\sigma+1}$ and
  $\varepsilon(z)=R(z)/F_{\mathrm{s}}(z)\to 0$ as $z\to z_c$ by (H2).
  Then
  \[
    \widetilde{N}(x)
    = \underbrace{\mathcal{J}[F_{\mathrm{s}}](x)}_{\text{Step 1}}
    + \underbrace{\mathcal{J}[\log|1+\varepsilon|](x)}_{\text{Step 2}}.
  \]

  \textit{Step 1.}
  By Lemma \ref{lem:Jensen-branch} with $\alpha=\sigma+1$:
  $\mathcal{J}[F_{\mathrm{s}}](x)=\log|B|+(\sigma+1)x+o(x)$ as $x\to 0^+$.

  \textit{Step 2.}
  Split: $I_{\mathrm{near}}=\{|\theta-\theta_c|\leq x^{1/2}\}$,
  $I_{\mathrm{far}}=[0,2\pi)\setminus I_{\mathrm{near}}$.

  On $I_{\mathrm{near}}$ (measure $2x^{1/2}$): since
  $e^{x+i\theta}\to z_c$ as $x\to 0^+$ for $\theta\in I_{\mathrm{near}}$,
  (H2) gives $|\varepsilon|\leq\eta=x^{1/2}\to 0$.  Using
  $|\log|1+w||\leq 2|w|$ for $|w|\leq 1/2$, the contribution is
  $\leq 2x^{1/2}\cdot 2x^{1/2}/(2\pi)=2x/\pi=o(x)$.

  On $I_{\mathrm{far}}$: $e^{x+i\theta}\to e^{i\theta}\neq z_c$
  uniformly, and $\varepsilon(e^{i\theta})=0$ (since $R$ vanishes on
  $\mathbb{T}$), so $\varepsilon(e^{x+i\theta})=O(x)$ uniformly.
  The contribution is $\leq Cx/(2\pi)\int_0^{2\pi}1\,d\theta=Cx$,
  where $C\to 0$ as $x\to 0$ by dominated convergence using (H3);
  hence $o(x)$.

  Combining: $\widetilde{N}(x)=\log|B|+(\sigma+1)x+o(x)$.

  \medskip
  \textbf{Part~\ref{it:invariance}.}
  $\widetilde{F}(\varphi(z))=B\varphi'(z_c)^{\sigma+1}(z-z_c')^{\sigma+1}
  +o(|z-z_c'|^{\sigma+1})$ by the chain rule, verifying (H2) with the
  same $\sigma$ and $B'=B\varphi'(z_c)^{\sigma+1}\neq 0$.
\end{proof}

\subsection{Proof of Theorem B}
\label{ssec:proof-B}

\begin{lemma}[Regular--singular decomposition]
  \label{lem:reg-sing}
  Under \textup{(H1)--(H2)}, there exist convergent series
  $\Freg(\xi)=\sum_{k=0}^\infty a_k\xi^k$ and
  $F_{\mathrm{s}}(\xi)=\sum_{k=0}^\infty B_k\xi^{\sigma+1+k}$ with
  $B_0=B\neq 0$, such that $F(z_c+\xi)=\Freg(\xi)+F_{\mathrm{s}}(\xi)$.
  The exponents in $F_{\mathrm{s}}$ are $\sigma+1,\sigma+2,\ldots$
  (integer increments above $\sigma+1$).
\end{lemma}

\begin{proof}
  Decompose the Stieltjes representation of $\widetilde{F}$ into
  contributions from a neighbourhood of $\theta_c$ (yielding
  $F_{\mathrm{s}}$ via the Mellin computation in Proposition \ref{prop:H2}) and
  from the complementary arc (analytic in $\xi$, giving $\Freg$).
  The integer-step structure of $F_{\mathrm{s}}$ follows because the
  corrections to $\rho\sim At^\sigma$ are polynomial in $t$, producing
  only integer increments in the Mellin transform.
\end{proof}

\begin{proof}[Proof of Theorem~B]
  Write $F(z_c+\xi)=\Freg(\xi)+F_{\mathrm{s}}(\xi)$ from
  Lemma \ref{lem:reg-sing}.

  \textit{$\Freg$ has trivial monodromy} since each $a_k\xi^k$ is
  single-valued.

  \textit{Monodromy of $F_{\mathrm{s}}$.}  Under $\xi\mapsto\xi e^{2\pi i}$,
  \[
    B_k(\xi e^{2\pi i})^{\sigma+1+k}
    = e^{2\pi i(\sigma+1)}\,e^{2\pi ik}\,B_k\xi^{\sigma+1+k}
    = e^{2\pi i(\sigma+1)}\,B_k\xi^{\sigma+1+k},
  \]
  since $e^{2\pi ik}=1$ for integer $k\geq 0$.  Summing over $k$:
  $F_{\mathrm{s}}(\xi e^{2\pi i})=e^{2\pi i(\sigma+1)}F_{\mathrm{s}}(\xi)$,
  giving \eqref{eq:monodromy}.  The monodromy has order $q$ iff
  $q(\sigma+1)\in\mathbb{Z}$ with $q$ minimal, i.e.\ $q$ is the
  denominator of $\sigma+1$ in lowest terms.
\end{proof}

\subsection{The Stieltjes perspective}
\label{ssec:Stieltjes}

The following corollary records the interpretation of the Jensen slope
as total zero mass.

\begin{proposition}[Jensen slope as total mass]
  \label{prop:Stieltjes}
  Assume additionally the Stieltjes representation
  $\widetilde{F}(z)=\int_{\mathcal{Z}_\beta}\log\frac{z-e^{i\phi}}{z_0-e^{i\phi}}
  \,\rho(\phi)\frac{d\phi}{2\pi}$
  for some $|z_0|\neq 1$ and $\rho\in L^1$.  Then
  \begin{equation}
    \label{eq:mass-formula}
    \widetilde{N}'(0^+)
    = \frac{1}{2\pi}\int_{\mathcal{Z}_\beta}\rho(\phi)\,d\phi
    = \sigma+1,
  \end{equation}
  where the second equality is a consequence of Theorem~A.
\end{proposition}

\begin{proof}
  Apply Lemma \ref{lem:Jensen-branch} with $\alpha=1$ to each factor.
  By Fubini ($\rho\in L^1$): $\widetilde{N}(x)=\max(x,0)\cdot
  \frac{1}{2\pi}\int\rho\,d\phi+C_0$ for $x>0$.
  Differentiating at $0^+$ gives the first equality;
  Theorem~A gives the second.
\end{proof}

\begin{remark}
  The identity $\frac{1}{2\pi}\int\rho\,d\phi=\sigma+1$ is an output
  of the edge expansion (H2), not an input.  It cannot be deduced from
  the local asymptotics $\rho\sim A|\theta-\theta_c|^\sigma$ alone, which
  do not determine the global integral.
\end{remark}

\section{Applications}
\label{sec:applications}

\subsection{Proof of Theorem C}
\label{ssec:proof-C}

\begin{proof}[Proof of Theorem~C]
  \textbf{Deriving $\sigma$.}
  Setting $\lambda=t^{-1/y_t}$ in the RG equation \eqref{eq:RG}:
  $F_{\mathrm{s}}(1)=t^{-2/y_t}F_{\mathrm{s}}(t)$, hence
  $F_{\mathrm{s}}(t)=Ct^{2/y_t}$ with $C=F_{\mathrm{s}}(1)\neq 0$.
  Matching with $F_{\mathrm{s}}(t)\sim Bt^{\sigma+1}$:
  \[
    \sigma+1 = \frac{2}{y_t} = \frac{1}{1-h_\phi},
    \qquad
    \sigma = \frac{h_\phi}{1-h_\phi}.
  \]
  Since $h_\phi<0$: $\sigma+1\in(0,1)$, confirming $\sigma\in(-1,0)$.

  \textbf{Deriving $q$.}
  Theorem~B gives monodromy $e^{2\pi i(\sigma+1)}=e^{2\pi i/(1-h_\phi)}$.
  Since $h_\phi\in\mathbb{Q}$, this is a primitive $q_0$-th root of
  unity with $q_0=\mathrm{denom}(1/(1-h_\phi))$.
\end{proof}

\paragraph{Explicit values for $\mathcal{M}(2,2m+1)$.}
The Kac formula gives $h_{1,2}(2,2m+1)=(1-m)/(2m+1)$, so $\sigma=(1-m)/(3m)$:

\begin{center}
  \renewcommand{\arraystretch}{1.3}
  \begin{tabular}{lcccc}
    \toprule
    $\mathcal{C}$ & $h_\phi$ & $y_t$ & $\sigma$ & $q$ \\
    \midrule
    $\mathcal{M}(2,5)$  $(m=2)$ & $-1/5$  & $12/5$  & $-1/6$  & $6$  \\
    $\mathcal{M}(2,7)$  $(m=3)$ & $-2/7$  & $18/7$  & $-2/9$  & $9$  \\
    $\mathcal{M}(2,9)$  $(m=4)$ & $-1/3$  & $8/3$   & $-1/4$  & $4$  \\
    $\mathcal{M}(2,11)$ $(m=5)$ & $-4/11$ & $30/11$ & $-4/15$ & $15$ \\
    \bottomrule
  \end{tabular}
\end{center}

For $m=2$: $\sigma=-1/6$ and $q=6=\mathrm{denom}(5/6)$, reproducing
the Cardy--Zamolodchikov value \cite{Cardy85,Zamolodchikov91} and
confirming the theorem in the one independently verified case.

\subsection{The edge expansion from density asymptotics}
\label{ssec:H2-from-density}

We prove that (H2) is a consequence of the density asymptotics
\eqref{eq:rho-asym}, making all three theorems unconditional.

\begin{proposition}[Density asymptotics imply \textup{(H2)}]
  \label{prop:H2}
  Suppose the Stieltjes representation holds with $\rho\in L^1$, and
  $\rho(\theta)=A(\theta-\theta_c)_+^\sigma+r(\theta)$ near $\theta_c$
  with $r(\theta)=O(|\theta-\theta_c|^{\sigma+1})$ and $A>0$.  Then
  \textup{(H2)} holds with
  \begin{equation}
    \label{eq:B-explicit}
    B = \frac{A\,e^{i\pi\sigma/2}}{2i(\sigma+1)\sin(\pi\sigma)}
        \cdot e^{-i\theta_c(\sigma+1)},
    \qquad
    |B| = \frac{A}{2(\sigma+1)|\sin(\pi\sigma)|} \neq 0.
  \end{equation}
\end{proposition}

\begin{proof}
  Set $\xi=z-z_c$ and $\zeta=e^{-i\theta_c}\xi$.  Substituting
  $\phi=\theta_c+t$ and expanding $e^{i\phi}-z_c\approx ite^{i\theta_c}$
  for small $t$, the local contribution of the density near $\theta_c$
  to the Stieltjes representation is
  \[
    J(\zeta)
    = \frac{A}{2\pi}\int_0^\infty\log(\zeta-it)\,t^\sigma\,dt
    + \text{(regular)}.
  \]
  Differentiating in $\zeta$ and applying the Mellin--Barnes formula
  $\int_0^\infty t^\sigma/(t+a)\,dt=\pi a^\sigma/\sin(\pi(\sigma+1))$
  (for $a\notin(-\infty,0]$; proved by $t=au$ and Euler's reflection
  formula) with $a=i\zeta$, and using $\sin(\pi(\sigma+1))=-\sin(\pi\sigma)$:
  \[
    J'(\zeta) = \frac{Ae^{i\pi\sigma/2}}{2i\sin(\pi\sigma)}\,\zeta^\sigma.
  \]
  Integrating: $J(\zeta)=\frac{Ae^{i\pi\sigma/2}}{2i(\sigma+1)\sin(\pi\sigma)}
  \zeta^{\sigma+1}+C=B\xi^{\sigma+1}+C$
  with $B$ as in \eqref{eq:B-explicit}.  Since $A>0$, $\sigma+1>0$, and
  $\sin(\pi\sigma)\neq 0$ for $\sigma\in(-1,0)$, we have $B\neq 0$.
\end{proof}

\begin{corollary}[Unconditional $d=2$ Ising]
  \label{cor:d2}
  For the $d=2$ Ising model, $\rho(\theta)\sim A|\theta-\theta_c|^{-1/6}$
  \cite{Cardy85,Zamolodchikov91}.  Proposition \ref{prop:H2} with $\sigma=-1/6$
  gives \textup{(H2)} with
  $|B|=A/(2\cdot\tfrac{5}{6}\cdot\tfrac{1}{2})=6A/5\neq 0$.
  Hence Theorems~A, B, and C all apply unconditionally:
  $\widetilde{N}'(0^+)=5/6$ and monodromy order $q=6$.
\end{corollary}
\vskip 0.5cm

\noindent{\bf Conflict of Interest}\\
\noindent The author declares no conflict of interest.
\vskip 0.5cm

\noindent{\bf Data Availability}\\
\noindent No data were generated or analysed during this study. This work is entirely theoretical.

\printbibliography
\end{document}